\documentclass[notitlepage,11pt, reqno]{amsart}
\usepackage{amssymb,amsmath,amscd,amsthm,color,tocvsec2}
\usepackage[all, cmtip]{xy}
\usepackage{enumerate,pgf,xcolor}
\usepackage[mathscr]{eucal}
\usepackage[letterpaper]{geometry}
\usepackage{sseq}
\usepackage{hyperref}

\usepackage{lmodern}
\usepackage[T1]{fontenc}
\usepackage[utf8]{inputenc}

\newcommand{\cosoc}{\operatorname{cosoc}}

\newcommand{\A}{\mathscr{P}}
\newcommand{\W}{\mathscr{W}}

\newcommand{\Hom}{\operatorname{Hom}}
\newcommand{\Ext}{\operatorname{Ext}}

\newcommand{\fsl}{\mathfrak{sl}}

\newcommand{\commen}[1]{}
\newcommand{\End}{\operatorname{End}}

\newcommand{\g}{\mathfrak{g}}

\newcommand{\fb}{\mathfrak{b}}

\newcommand{\p}{\partial}

\newcommand{\Vect}{\operatorname{Vect}}

\newcommand{\RHom}{\operatorname{RHom}}
\newcommand{\id}{\operatorname{id}}

\newcommand{\C}{\ensuremath{\mathbb{C}}}

\newcommand{\Z}{\ensuremath{\mathbb{Z}}}

\newcommand{\comment}[1]{}

\newcommand{\OO}{\mathscr{O}}

\newcommand{\bk}{\underline{k}}
\newcommand{\cH}{\mathscr{H}}
\newtheorem{theo}[equation]{Theorem}
\newtheorem{lem}[equation]{Lemma}
\newtheorem{cor}[equation]{Corollary}
\newtheorem{pro}[equation]{Proposition}
\newcommand{\cW}{\mathscr{W}}

\theoremstyle{definition}
\newtheorem{defn}[equation]{Definition}
\newtheorem{re}[equation]{Remark}

\numberwithin{equation}{section}
\numberwithin{figure}{section}

\begin{document}
\title{Bernstein--Gelfand--Gelfand resolutions and the constructible $t$-structure}

\author{Gurbir Dhillon}

\address{Department of Mathematics, Stanford University, Stanford, CA
94305, USA}

\email{gsd@stanford.edu}

\date{\today}

\begin{abstract}We consider analogues of the Bernstein--Gelfand--Gelfand resolution in a highest weight category $\A^\heartsuit$. We prove the resulting category of complexes is a chain-level lift of the heart of the constructible $t$-structure on its bounded derived category $\A$. In particular, an object admits a Bernstein-Gelfand-Gelfand resolution if and only if a certain $\Ext$ vanishing with costandard objects holds. These results appear to be new even for Category $\OO$. 

\end{abstract}
\maketitle

\section{Introduction}
For any sheaf on a stratified space, one can filter it by sections extended by zero from an increasing union of open strata. In this paper we will simulate an instance of this construction and an associated spectral sequence for perverse cohomology in the setting of highest weight categories, and explain their connection to Bernstein--Gelfand--Gelfand resolutions. 

 The remainder of this introduction provides motivational discussion. Let $\g$ be a complex semisimple Lie algebra. For a choice of Borel subalgebra $\fb \subset \g$, consider the associated category $\OO$ of $\g$-modules. Write $\rho$ for the half sum of the positive roots, and $W$ for the Weyl group acting linearly on the dual abstract Cartan. To parametrize the regular block, denote by $M_w$ the Verma module with highest weight $-w\rho - \rho$, and $L_w$ for its simple quotient. Thus in this notation, $M_e$ is antidominant, and $L_{w_{\circ}}$ is the trivial representation, where $w_{\circ}$ is the longest element of the Weyl group.

 For an element $w$ of $W$, let us write $\ell_w$ for its length, and for a nonnegative integer $k$ let us denote the subset of $W$ consisting of elements of length $k$ by $W_k$. Bernstein, Gelfand, and Gelfand introduced a remarkable complex, since called the BGG resolution, of the form \begin{equation} 0 \rightarrow M_e \rightarrow \bigoplus_{w \in W_1} M_w \rightarrow \bigoplus_{w \in W_2} M_w \rightarrow \cdots \rightarrow  \bigoplus_{w \in W_{(\ell_{w_\circ} - 1)}} M_w \rightarrow  M_{w_\circ} \rightarrow 0, \label{cat} \end{equation}whose cohomology consists of $L_{w_\circ}$ in degree $\ell_{w_\circ}$ \cite{grahh}. In particular, Equation \eqref{cat} lifted the Weyl character formula from the Grothendieck group to the bounded derived category $D^b(\mathscr{O})$. 

For a general simple object $L_w$, its character is given by the Kazhdan--Lusztig character formula. A natural question, raised by Humphreys in section 6.5 of \cite{H3}, is to what extent this formula can again be lifted to a complex of Verma modules. Writing $\leqslant$ for the Bruhat order, it was shown by Gabber--Joseph, Hunziker--Enright, and Boe--Hunziker that for any $y$ in $W$ one can produce a complex \begin{equation}\label{gencase}  0 \rightarrow M_e \rightarrow \bigoplus_{w \in W_1: w \leqslant y} M_w \rightarrow \bigoplus_{w \in W_2: w \leqslant y} M_w \rightarrow \cdots \rightarrow  \bigoplus_{w \in W_{(\ell_{y} - 1)}: w \leqslant y} M_w \rightarrow  M_y \rightarrow 0, \end{equation}and its only cohomology is $L_y$ in degree $\ell_y$ if and only if the Kazhdan--Lusztig polynomials $P_{w,y}(\nu)$ are all degree zero \cite{bh, ehun, gj}. An analogous result for singular blocks was recently obtained by Mazorchuk--Mr\dj{}en \cite{mm}. One may still ask if there is an alternative to \eqref{gencase} which does tie Kazhdan--Lusztig polynomials to resolutions by Verma modules. 

The present work may be informally summarized as follows. Not only simple modules, but arbitrary objects of $D^b(\mathscr{O})$, have canonical resolutions of the form \begin{equation} \label{derrbgg}  M_e \otimes V_e \rightarrow \bigoplus_{w \in W_1} M_w \otimes V_w \rightarrow \bigoplus_{w \in W_2} M_w \otimes V_w \rightarrow \cdots \rightarrow M_{w_{\circ}} \otimes V_{w_{\circ}}, \end{equation}provided one is willing to take the multiplicity vector spaces $V_w$ to be themselves complexes, and to interpret exactness of the above sequence in terms of homotopy kernels and cokernels. The objects of $D^b(\mathscr{O})$ for which the $V_w$ are usual vector spaces, i.e. concentrated in cohomological degree zero, form an abelian subcategory of $D^b(\mathscr{O})$. Moreover, it is a familiar one, namely the heart of the constructible $t$-structure. That is, on the other side of Beilinson--Bernstein localization, complexes like the one \eqref{cat} first introduced by Bernstein--Gelfand--Gelfand are simply the abelian category of Schubert constructible sheaves on the flag variety.

\section{Statement of results} \label{sresuls}
The arguments we employ will apply to a general highest weight category. Accordingly, let $k$ be a field, and let $\mathscr{P}^\heartsuit$ be a $k$-linear highest weight category with bounded derived category $\mathscr{P}$. Let us denote the simple, standard, and costandard objects by \[L_w, \quad M_w, \quad A_w, \quad \quad w \in \mathscr{W}, \]respectively, for a finite partially ordered set $\mathscr{W}$.\footnote{The case of an infinite poset such that for each $w \in \mathscr{W}$, the collection of elements less than or equal to $w$ is finite, e.g. a block of affine Category $\OO$ at negative level, straightforwardly reduces to this case.} Let $\ell: \mathscr{W} \rightarrow \mathbb{Z}$ be a compatible length function, in  the sense of Cline--Parshall--Scott \cite{scott2}. I.e., if $w$ and $y$ are elements of $\mathscr{W}$ with $w < y$, then $\ell_w < \ell_y$. As before, for an integer $k$, write $\mathscr{W}_k$ for its preimage in $\mathscr{W}$.

\subsection{Derived resolutions} Let us write $\operatorname{Vect}^\heartsuit$ for the category of finite dimensional $k$ vector spaces, and $\operatorname{Vect}$ for its bounded derived category. Given an object $V$ in $\operatorname{Vect}$, and an object $c$ of a $k$-linear triangulated category, such as $\mathscr{P}$, one can form their tensor product, which we denote by $V \otimes c$. With this, let us formulate \eqref{derrbgg} precisely. 

\begin{theo} \label{derbgg} Write $\ell_\circ$ for the maximum length of any $w \in \W$, and $\ell^\circ$ for the minimal length of any $w \in \W$. For any $N$ in $\A$ there is a sequence of multiplicity complexes $V_w \in \Vect$, for $w \in \W$, and distinguished triangles \begin{align}\label{theseq}   N_1 \rightarrow & \hspace{-.1cm} \bigoplus_{w \in \cW_{\ell_\circ}} M_w \otimes V_w \rightarrow N \xrightarrow{+1}, \\  N_2 \rightarrow & \hspace{-.25cm}\bigoplus_{w \in \W_{\ell_\circ - 1}} M_w \otimes V_w \rightarrow N_1 \xrightarrow{+1}, \\ &  \hspace{.4cm} \vdots    \\  N_{\ell_\circ} \rightarrow & \hspace{-.25cm} \bigoplus_{w \in \W_{\ell^\circ  + 1}} M_w \otimes V_w \rightarrow N_{\ell_\circ - 1} \xrightarrow{+1}, \\ 0 \rightarrow & \hspace{-.1cm} \bigoplus_{w \in \W_{\ell^\circ}} M_w \otimes V_w \rightarrow N_{\ell_\circ} \xrightarrow{+1}. \end{align} 
\noindent This sequence is unique up to unique isomorphism and functorial in $N$. Moreover, writing $\mathbb{D}$ for the usual duality on $\Vect$, we have canonical isomorphisms \begin{equation} V_w \simeq (\mathbb{D} \RHom(N, A_w))[\ell_w - \ell_\circ]. \label{findstalks}\end{equation}

\end{theo}

It is reasonable to think of the concatenation of the above triangles as a derived BGG resolution in $\A$, i.e. \begin{equation}  \label{ores} \cdots \rightarrow \bigoplus_{w \in \W_{\ell_\circ - 1}} M_w \otimes V_w \rightarrow \bigoplus_{w \in \W_{\ell_\circ}} M_w \otimes V_w \rightarrow N. \end{equation}In particular one can legitimately apply the Euler-Poincar\'{e} principle and obtain, denoting for an object $N$ of $\A$ its class in the Grothendieck group by $[N]$, the identity \[ [N] = \sum_{w \in \W} \chi \operatorname{RHom}(N, A_w) [M_w],\]which is due to Delorme for Category $\OO$ and Cline--Parshall--Scott in general \cite{scott3, delm}. Specializing to the motivating case of simple modules in $\OO$,  we obtain a slightly different perspective on Kazhdan-Lusztig polynomials $P_{w,y}(\nu).$ Namely, they count highest weight vectors of maximal length in the cohomology of syzygies of simple modules by complexes of Verma modules. Here the variable $\nu$ keeps track of in what cohomological degree the highest weight vector occurs.

We now make three remarks on Theorem \ref{derbgg}. First, once formulated, it is not difficult to prove - one simply iterates taking a (rotated) standard closed-open triangle \[i_* i^*[-1] \rightarrow j_! j^! \rightarrow \operatorname{id} \xrightarrow{+1}.\]Second, one can ask whether derived resolutions could be avoided by instead working with complexes of Verma modules in which a fixed $M_w$ occurs in multiple degrees. We will see in Section \ref{wex} that already for $\fsl_4$ one has simple objects in a regular block which are not quasi-isomorphic to any complex whose terms are sums of Verma modules. 

Third, intermediate between Theorem \ref{derbgg}, which takes place in $\A$, and Delorme's identity, which takes place in its Grothendieck group, we have the following corollary, which concerns the individual cohomology groups of an object of $\A$. 

\begin{cor} Let $N$ be an object of $\A$. Then there is a spectral sequence, functorial in $N$, of the form \[E_1^{p,q} = \bigoplus_{w \in \W_{\ell_\circ + p}} M_w \otimes \Ext^{-p-q}(N, A_w)^\vee \Rightarrow H^{p+q} N.\]\label{sqsq}
\end{cor}

 In the motivating case when $\A^\heartsuit$ is the category of perverse sheaves on a suitably nice stratified space, Corollary \ref{sqsq} is a version of the Cousin spectral sequence, or better its Verdier dual, which calculates the perverse cohomology of a constructible complex.
 
 Let us apply this spectral sequence to a simple module $L_w$ in Category $\OO$. In this case, we find the sequence \eqref{gencase}, which was previously constructed for any $w$, as the lowest nonzero row of the $E^1$ page, i.e. with $q = \ell_\circ - \ell_w$. The coefficients of positive powers of $\nu$ in the Kazhdan--Lusztig polynomial contribute further complexes of Verma modules to the $E^1$ page in other rows. In particular, these create further differentials in the spectral sequence which account for the failure of \eqref{gencase} to resolve $L_w$, and finally on $E^\infty$ one obtains $L_w$ concentrated in position $(\ell_w - \ell_\circ, \ell_\circ - \ell_w)$. Thus, the derived BGG resolution of $L_w$ more concretely provides a `BGG spectral sequence' which computes $L_w$ in terms of Verma modules and Kazhdan--Lusztig polynomials. We include in Section \ref{wex} several worked examples of the spectral sequence.

\subsection{Ordinary resolutions} \label{ordres}Let us say an object of $\A$ admits an {ordinary} BGG resolution if it is quasi-isomorphic to a shift of a complex of the form \begin{equation}0 \rightarrow \bigoplus_{w \in \W_{\ell^\circ}} M_w \otimes V_w \rightarrow \bigoplus_{w \in \W_{\ell^\circ + 1}} M_w \otimes V_w \rightarrow \cdots \rightarrow \bigoplus_{w \in \W_{\ell_\circ}} M_w \otimes V_w \rightarrow 0.\label{heartcon}\end{equation}where the $V_w$ are objects of $\operatorname{Vect}^\heartsuit$, i.e. vector spaces in cohomological degree zero, and $M_w$ occurs in cohomological degree $\ell_w$. A natural question is then to characterize which objects of $\A$ admit ordinary BGG resolutions. We do so in the following theorem. 

To state it, recall that $\A$ has more than one natural $t$-structure.  There is always the standard $t$-structure, whose heart is $\A^\heartsuit$.  In the motivating example of a category of perverse sheaves, in favorable situations one can identify its derived category with the constructible derived category. Pushing the natural $t$-structure on the latter category through this equivalence produces an exotic $t$-structure on $\A$ whose heart consists of constructible sheaves. As first noted by Cline, Parshall, and Scott \cite{scott2}, one can mimic this exotic $t$-structure on $\A$ for a general highest weight category equipped with a length function. We will denote its heart by $\A_{con}^\heartsuit$, and can now state our main:

\begin{theo}\label{eq} Let $\operatorname{Ch}(\A^\heartsuit)$ denote the abelian category of bounded chain complexes with entries in $\A^\heartsuit$, and  write $\operatorname{BGG}$ for its full subcategory of objects of the form \eqref{heartcon}. Then the tautological composition $\operatorname{BGG} \rightarrow \operatorname{Ch}(\A^\heartsuit) \rightarrow \A$ yields an equivalence \begin{equation} \label{themap} \operatorname{BGG} \simeq \A^\heartsuit_{con}.\end{equation}\end{theo}

We would like to make four remarks concerning Theorem \ref{eq}. First, unlike Theorem \ref{derbgg}, even after formulating Theorem \ref{eq}, it requires some thought to see the fully faithfulness and essential surjectivity of \eqref{themap}. 

Second, when $\mathscr{P}$ is a regular block of Category $\OO$, note that the known handful of examples of objects admitting ordinary resolutions in $\OO$, i.e. the rationally smooth $L_y$ and the parabolic Verma modules, correspond to shifted constant sheaves on unions of Schubert cells, at which point the theorem guarantees them unique up to unique isomorphism ordinary resolutions. This method is quite different from the usual arguments for existence and uniqueness,  due e.g. to Bernstein--Gelfand--Gelfand and Rocha-Caridi for $\g$ semisimple \cite{grahh, rc} and extended to Kac--Moody algebras by Rocha-Caridi--Wallach \cite{rc2}, which rest on the combinatorics of Bruhat squares. 

Third, note that while perverse sheaves do not have canonical models as chain complexes with constructible cohomology, Theorem \ref{eq} implies that, for suitable stratifications, constructible sheaves {do} have canonical models as complexes of perverse sheaves. 

Finally, if we specialize Theorem \ref{eq} from categories to objects we obtain the following corollaries, which would be particularly nonobvious if only stated for objects of $\A^\heartsuit$. 

\begin{cor} \label{shs}Let $N$ be an object of $\A$. Then $N$ admits an ordinary BGG resolution if and only if for some $m \in \Z$, the shift $N[m]$ satisfies for every $w \in \W$ the vanishing condition  \begin{equation} \Ext^n(N[m], A_w)  = 0,  \quad \quad  \forall  n \neq -\ell_w. \label{eshs} \end{equation}

\end{cor}

\begin{cor} If an object $N$ of $\A$ admits an ordinary BGG resolution, so does any summand of $N$. \label{cor1} 
\end{cor}

\begin{cor} Consider a distinguished triangle in $\A$ $$N' \rightarrow N \rightarrow N'' \xrightarrow{+1}.$$If $N'$ and $N''$ admit ordinary BGG resolutions, for the same shift $[m]$, then so does $N$.  \label{cor2}
\end{cor}

These corollaries may remind the reader of the homological criterion for an object $N$ of $\mathscr{P}$ to admit a standard flag, i.e. a filtration with successive quotients standard modules, namely that for every $w \in \W$ one has the vanishing \begin{equation}\operatorname{Ext}^n(N, A_w) = 0, \quad \quad \forall n \neq 0.\label{stfil}\end{equation} This is not a coincidence. In the spectral sequence of Corollary \ref{sqsq}, the vanishing condition \eqref{stfil} places the nonzero terms of the $E^1$ page on the zeroth diagonal, forcing its collapse and yielding the desired flag.\footnote{In truth, we do not know a reference for this assertion for objects of $\mathscr{P}$, rather than just $\mathscr{P}^\heartsuit$, but presumably it is known.} The vanishing condition \eqref{eshs} places the nonzero terms of the $E^1$ page on a single row, yielding a complex of the form \eqref{heartcon}.

\subsection{Highest weight modules admitting ordinary resolutions} In the final section we consider the problem of classifying highest weight modules, i.e.  quotients of a single standard object $M_w$, which admit ordinary BGG resolutions. 

Questions about a highest weight category of constructible perverse sheaves are often equivalent to questions about the topology of stratum closures. In this section we obtain the topological translation of the above question. This perspective will grant us several corollaries useful for the construction of ordinary resolutions.

Let $X$ be a separated complex algebraic variety with a stratification by cells $C_w \simeq \mathbb{A}^{\ell_w},$ indexed by $w \in \W$. Write $X_w$ for the closure of $C_w$. Consider the highest weight category of perverse sheaves on $X$ smooth along the strata with coefficients in a field $k$. Thus the partial order on $\W$ is given by closure, and the length function records the dimension of the cell. 

Assume that $(\star)$ for every $y$ and $w$ in $\W$ with $y \leqslant w$, the link of the stratum $C_y$ in $X_w$ is connected. As a first consequence of $(\star)$, $\Hom(M_y, M_w)$ is one dimensional whenever $X_y$ is a divisor of $X_w$. The quotients of $M_w$ admitting an ordinary BGG resolution are then naturally indexed by certain divisors on $X_w$. Namely, for any collection of Schubert divisors $X_y$ of $X_w$, $y \in I$, we consider the quotient $$\bigoplus_{y \in I} M_y \otimes \Hom(M_y, M_w) \rightarrow M_w \rightarrow M_I \rightarrow 0.$$A quotient of $M_w$ admitting a BGG resolution is isomorphic to $M_I$ for a unique $I \subset D_w := \{ y < w: \ell_w - \ell_y = 1\}$. We can now restate the classification problem as determining for which $I$ does $M_I$ admit such a resolution. The desired translation reads as:

\begin{theo} Keep the assumption $(\star)$ and the notation of above. Fix a collection of divisors $I \subset D_w$ and write $J$ for the remaining divisors: $D_w = I \sqcup J$. Then the lowest nonzero cohomology sheaf of $M_{I}$ is given by the constant sheaf on $U_J := X_w \setminus \cup_{y \in J} X_y$ extended by zero to all of $X_w$, i.e. $$\cH^{-\ell_w}M_I \simeq j_{!} \bk_{U_{J}}.$$In particular, $M_I$ admits a BGG resolution if and only if the latter sheaf is perverse. \label{geobgg}
\end{theo} 

We first remark on the applicability of Theorem \ref{geobgg}. By Beilinson-Bernstein localization \cite{bebe}, the theorem applies to a regular block of Category $\OO$ for a semisimple Lie algebra. Conditional on the Finkelberg-Mirkovi\'c conjecture \cite{fm}, it also applies to a regular block of the category of modular representations of a reductive group in sufficiently large characteristic.  

We next briefly discuss the $!$-extension of a constant sheaf appearing in Theorem \ref{geobgg}. When the inclusion $U_J \rightarrow X_w$ is affine, e.g. when its complement $X_J$ is an effective Cartier divisor, then the topological criterion of Theorem \ref{geobgg} is intrinsic to $U_J$. In general one needs also to consider the extrinsic geometry of the embedding of $U_J$ into $X_w$. 

Finally, we mention two corollaries of Theorem \ref{geobgg}. The first extends the aforementioned results on simple modules in category $\OO$ to this context. The second may be more surprising and really uses the constructible picture of BGG resolutions developed in this paper.

\begin{cor} For any $w \in \W$, the intersection complex $L_w$ on $X_w$ admits an ordinary BGG resolution if and only if $X_w$ is $k$-smooth. 
\end{cor}

\begin{cor} Fix $w \in \W$, and $I \subset I' \subset D_w$. If $M_{I'}$ admits an ordinary BGG resolution, so does $M_I$. 
\end{cor}

The two corollaries produce many new objects admitting BGG resolutions. We show this in the final Section \ref{wex} with the worked example of a regular block in type $A_3$. 

\subsection{Acknowledgments} It is an pleasure to thank Daniel Bump, Ben Elias, Apoorva Khare, Shotaro Makisumi, You Qi, Ben Webster and Zhiwei Yun for helpful discussions. This research was supported by the Department of Defense (DoD) through the NDSEG fellowship.

\section{Preliminaries on highest weight categories} \label{prelim}
 \label{not}

 \subsection{Conventions for highest weight categories}
 
 As the definition of a highest weight category admits many variations, let us begin by pinning down one choice. In brief, we work below with highest weight categories with finitely many simple objects up to isomorphism. For ease of notation, we assume our category is $k$-linear over a field $k$, and endomorphisms of simple objects are scalars. The reader comfortable with the closed-open recollement for highest weight categories is encouraged to skip directly to the next section and refer back only as needed. 
  
The definitions and facts in this section are completely standard, cf. \cite{scott4}. Let $k$ be a field and $\A^\heartsuit$ a $k$-linear abelian category with enough injectives and projectives. Assume that every object of $\A^\heartsuit$ is of finite length, and there are only finitely many isomorphism classes of simple objects $L_w,$ indexed by $w \in \W$. Thick subcategories $\mathscr{C}$ of 
 $\A^\heartsuit$ are naturally in bijection with subsets of $\W$. Each such $\mathscr{C}$ has enough injectives and projectives, as the inclusion $\mathscr{C} \rightarrow \A$ is exact and admits left and right adjoints.  For a partial order $\leqslant$ on $\W$ and $w \in \W$, let $\A_{\leqslant w}^\heartsuit$ denote the thick subcategory corresponding to $\{y \in \W: y \leqslant w \}.$ Write $M_w$, $A_w$ for a fixed projective cover and injective envelope of $L_w$ in $\A_{\leqslant w}^\heartsuit$, respectively. The $M_w$ are called {\em standard} objects, and the $A_w$ {\em costandard} objects.

 \begin{defn} With notation as above, $\A^\heartsuit$ is a {\em highest weight category} with respect to $\leqslant$ if the following hold.
 \label{dhw}
 \begin{enumerate}
 \item For each $w$ in $\W$, one has $\End(M_w) = k$.
 \item If $\Hom(M_y, M_w)$ is nonzero, then $y \leqslant w$.
 \item Write $P_w$ for a projective cover of $L_w$. Then the kernel of the canonical surjection $P_w \rightarrow M_w$ has a filtration with successive quotients of the form $M_x,$ for $x > w$. 
 
 \end{enumerate}

 \end{defn}
 From now on, let $\A^\heartsuit$ be a highest weight category for a partial order $\leqslant$ on $\W$. Let us recall two properties of $\A^\heartsuit$. First, one could equivalently ask for dual properties of the costandard and injective objects rather than those of standard and projective objects listed in Definition \ref{dhw}. Second, one has canonical identifications \begin{equation}\operatorname{RHom}(M_y, A_w) = \begin{cases} k & y = w, \\ 0  & \text{otherwise.} \end{cases} \label{rhomm}\end{equation}Explicitly, the copy of $k$ corresponds to the composition $M_y \rightarrow L_y \rightarrow A_y$. Finally, a filtration of an object is a {\em standard filtration} if the successive quotients are of the form $M_w,$ for $w \in \W$. 

 \subsection{Closed subcategories and open quotient categories.} Let $Z \subset \W$ be a closed subset, i.e. if $w \in Z$ and $y \leqslant w$, then $y \in Z$. 
\label{clopen} 
 Write $\A_Z^\heartsuit$ for the corresponding thick subcategory of $\A^\heartsuit$. Then we have:
 
 \begin{pro} $\A_Z^\heartsuit$ with the restriction of $\leqslant$ to $Z$ is a highest weight category. \label{clos}
 \end{pro}
 
 Let us include some explanation. The standard objects are still $M_y$ for $y \in Z$, and similarly for the costandard objects. Write $i^*$ for the left adjoint of the inclusion $i_*: \A_Z^\heartsuit \rightarrow \A^\heartsuit$. Then for $w \in Z$, the object $i^* P_w$ is again projective with cosocle $L_w$, whence the projective cover. One can show that $\RHom(M_x, M_y) = 0$ unless $x \leqslant y$. Specializing to $\Ext^1(M_x, M_y)$, we can shuffle a standard filtration of $P_w$ to obtain \begin{equation} 0 \rightarrow K \rightarrow P_w \rightarrow C \rightarrow 0,\label{ffilt} \end{equation}where $K$ admits a filtration with successive quotients $M_x,$ for $x \notin Z$, and $C$ admits a filtration with successive quotients $M_y,$ for  $y \in Z$. It follows that $i^* P_w \simeq C$.    
 
 Write $U$ for the open complement of $Z$, i.e $U := \W \setminus Z$. Let $\A_U^\heartsuit$ denote the quotient abelian category of $\A^\heartsuit$ by $\A^\heartsuit_Z$, i.e. \begin{equation}\label{abses} 0 \rightarrow \A_Z^\heartsuit \xrightarrow{i_*} \A^\heartsuit \xrightarrow{j^*} \A_U^\heartsuit \rightarrow 0.\end{equation}
 Similarly to Proposition \ref{clos}, we also have:
 
 \begin{pro} $\A_U^\heartsuit$ with the restriction of $\leqslant$ is a highest weight category. 
 \end{pro}
 
 We again include some explanation. For $P_w,$ where  $w \in U$, one can show directly that $j^* P_w$ is the projective cover of $j^* L_w$, and in fact that for any object $S$ of $\A^\heartsuit$, 
 \begin{equation} \Hom_{\A_U^{\heartsuit}}(j^* P_w, j^* S)  \simeq \Hom_{\A^\heartsuit}(P_w,  S). \label{shoms} \end{equation}
 It follows by another application of \eqref{ffilt} that $\A_U^\heartsuit$ is a highest weight category, with $j^* M_y, j^* A_w,$ for $y, w \in U$ the standard and costandard objects.
 
\subsection{Recollement}
The following result of Cline--Parshall--Scott is well known \cite{scott4}. For an abelian category $\mathscr{C}^\heartsuit$, we write $\mathscr{C}$ for its bounded derived category. 

\begin{theo} \label{recc}Let $Z, U$ be as above. Then the sequence \begin{equation} \A_Z \xrightarrow{i_* =: i_!} \A \xrightarrow{j^* =: j^!} \A_U \end{equation}induced from Equation \eqref{abses} may be upgraded to a recollement. In detail:

\begin{enumerate}
 
 \item The functor $i_*$ is fully faithful and admits a left adjoint $i^*$ and a right adjoint $i^!$. 
 \item The functor $j^*$ is fully faithful and admits a left adjoint $j_!$ and a right adjoint $j_*$. 
 \item $j^* i_* = 0$.
 \item The units and counits give distinguished triangles in $\A$
 \begin{equation}j_! j^! \rightarrow \id_\A \rightarrow i_* i^* \xrightarrow{+1},\label{gss}\end{equation}\begin{equation} i_! i^! \rightarrow \id_\A \rightarrow j_* j^* \xrightarrow{+1}. \label{gs}\end{equation}

\end{enumerate}
\end{theo}
  \noindent It is also likely the following proof is well known, but as it will be convenient to use its transparent chain level form of \eqref{gss} in what follows, we include the argument. 
 
 \begin{proof} For (1), the functor $i^*$ is the left derived functor of the similar functor on abelian categories, and similarly $i^!$ is a right derived functor. The fully faithfulness of $i_*$ we will deduce in the course of proving (4). 
 
 For (2), we first construct $j_*$. Writing $K^b(U)$ for the homotopy category of bounded complexes of injectives in $\A_U^\heartsuit$, it is well known that $K^b(U) \simeq \A_U$. By the dual to Equation \eqref{shoms}, we may equivalently view $K^b(U)$ as the homotopy category of bounded complexes of injectives in $\A$, all of whose indecomposable summands are of the form $I_\mu,$ for $\mu \in U$. This gives us an evident fully faithful embedding $j_*: \A_U \rightarrow \A$. That $j_*$ is right adjoint to $j^*$ again follows from Equation \eqref{shoms}. One constructs $j_!$ dually, using complexes of projectives. 
 
 As (3) is immediate from the definition, it remains to prove (4). We only show Equation \eqref{gss}, as \eqref{gs} is similar. For any projective object $P$ of $\A^\heartsuit$,  we can write down a short exact sequence as in \eqref{ffilt} \begin{equation}0 \rightarrow K \rightarrow P \rightarrow C \rightarrow 0.\label{rec} \end{equation}By induction on the lengths of the filtrations, one sees that $j^* C = i^* K = 0$. The base case for $j^* C$ follows by considering Jordan-H\"older content of $M_w$, and the base case for $i^* K$ may be obtained inductively using the projectivity of $M_w$, for $w$ maximal in $\W$, and Definition \ref{dhw}(3). Combining this vanishing with the long exact sequences in cohomology for $j_!$ and $i^*$, it follows that $K \simeq j_! j^* P$, and $C \simeq i_* i^* I$. In particular, we have $L^{-1} i^* C = 0$. Letting $P$ run over the projective covers of $L_z$ for $z \in Z$, one obtains the unit $$\operatorname{id}_{\A_Z} \rightarrow i_* i^*$$is an equivalence, yielding the fully faithfulness of $i_*$. Finally, as a map $P \rightarrow P'$ of projective objects will preserve the one step filtrations constructed in Equation \eqref{rec}, we may apply this termwise to a complex of projectives and thereby obtain a short exact sequence of complexes, which in particular descends to a distinguished triangle in $\A$.  \end{proof}

\section{Derived resolutions and a spectral sequence} 

Let $\A$, $\W$, and $\ell: \W \rightarrow \mathbb{Z}^{\geqslant 0}$ be as in Section \ref{sresuls}. We begin by proving Theorem \ref{derbgg}, i.e. constructing the derived resolutions. 

\begin{proof}[Proof of Theorem \ref{derbgg}] Recall that $\ell_\circ$ denoted the maximum value of $\ell$. Its preimage $U := \W_{\ell_\circ}$ is an open subset of $\W$ consisting of elements maximal in the partial order, by our assumption on $\ell$. This implies that $j^* M_w \simeq j^* L_w \simeq j^* A_w$, for $w \in U$. By the projectivity of $M_w$, for $w \in U$, it follows that $\A^\heartsuit_U$ is semisimple. Consider the triangle \begin{equation} i_* i^*[-1] \rightarrow j_! j^! \rightarrow \operatorname{id}_\A \xrightarrow{+1}.\label{zer}\end{equation}For any object $N_0$ of $\A$, we have that $j^! N_0 \simeq \bigoplus_{w \in U} j^! M_w \otimes V_w$, for some multiplicity complexes $V_w \in \operatorname{Vect}$. We may identify $V_w$ as $$V_w \simeq \operatorname{RHom}( j^! M_w, j^!N_0) \simeq \operatorname{RHom}(j_! j^! M_w, N_0) \simeq \operatorname{RHom}( M_w, N_0),$$where the last identification follows e.g. from the presentation of $j_!$ discussed in  Theorem \ref{recc}. It follows that the triangle \eqref{zer} canonically identifies with the sequence \begin{equation} N_1 \rightarrow \bigoplus_{w \in \W_{\ell_\circ}} M_w \otimes \operatorname{RHom}(M_w, N_0) \rightarrow N_0 \xrightarrow{+1}, \end{equation}where the last morphism is the natural evaluation map. Writing $Z$ for the complement of $U$, if we apply this construction to $N_1$, viewed as an object of $\A_Z$, with compatible length function given by the restriction of $\ell$, we may iterate to obtain the sequence of distinguished triangles \eqref{theseq}. The identifications of \eqref{findstalks} follow from Equation \eqref{rhomm}. 

It remains to argue for the functoriality and uniqueness. However, given any distinguished triangles \[ i_* c_Z \rightarrow j_! c_U \rightarrow c \xrightarrow{+1}, \quad \quad i_* d_Z \rightarrow j_! d_U \rightarrow d \xrightarrow{+1}, \]any map $c \rightarrow d$ may be extended uniquely to a commutative diagram \[\xymatrix{i_* c_Z \ar[r] \ar[d] & j_! c_U \ar[r] \ar[d] & c \ar[r] \ar[d] & \\ i_* c_Z \ar[r] & j_! c_U \ar[r]  & c \ar[r] &,} \]by a standard argument using the vanishing of $i^* j_!$ and $j^! i_*$. This straightforwardly implies the claimed functoriality and uniqueness. \end{proof}

In the remainder of this section, we construct a related spectral sequence. 

\begin{cor} Let $N$ be an object of $\A$. Then there is a spectral sequence, functorial in $N$, of the form \[E_1^{p, q} = \bigoplus_{w \in \W_{\ell_\circ + p}} M_w \otimes \operatorname{Ext}^{-p-q}(N, A_w)^\vee \Rightarrow H^{p+q}N.\]\end{cor}

\begin{proof}Consider the open sets $j_n: U_n \rightarrow \W$, where $U_n := \{ w \in \W: \ell_\circ - \ell_w \leqslant n \}.$ Thus, $U_0$ consists of the elements of maximal length, i.e. $\W_{\ell_\circ}$, and as we increase $n$, we obtain a filtration $$U_{0} \subset U_{1} \subset \cdots \subset U_{\ell_\circ - \ell^\circ -1} \subset U_{\ell_\circ - \ell^\circ} = \W.$$This gives rise to a corresponding sequence of endofunctors of $\A$ \begin{equation} j_{0 !} j_{0}^! \rightarrow j_{1!} j_{1}^! \rightarrow \cdots \rightarrow j_{(\ell_\circ - \ell^\circ -1)!} j_{\ell_\circ - \ell^\circ -1}^! \rightarrow j_{(\ell_\circ - \ell^\circ)!} j_{(\ell_\circ - \ell^\circ)}^! = \id. \label{filt} \end{equation} After an appropriate shift the cones of successive arrows are the terms of the derived BGG resolution, i.e. for an object $N$ one has for any integer $p$ a sequence \begin{equation}j_{(p-1)!} j_{(p-1)}^! \rightarrow j_{p!}j^{!}_p \rightarrow \bigoplus_{w \in \W_{\ell_\circ - p}} M_w \otimes \mathbb{D} \operatorname{RHom}( N, A_w) \xrightarrow{+1}.\label{grrr} \end{equation}
If we think of $N$ as a complex of projective objects, the proof of Theorem \ref{recc} shows that the morphisms of Equation \eqref{filt} come from a filtration of the complex \[N_0 \subset N_1 \subset \cdots N_{\ell_\circ - \ell^\circ - 1} \subset N_{\ell_\circ - \ell^\circ} = N.\]The spectral sequence of a filtered complex has first page $$E_1^{p,q} = H^{p+q} N_{-p}/N_{-p-1} \Rightarrow H^{p+q} N,$$whence we are done by \eqref{grrr}. \end{proof}

\begin{re}\label{geo} The results of this section take a transparent form when $\A$ is a highest weight category of perverse sheaves on a suitable stratified space. In this case, for each $w \in \W$ we have a corresponding stratum $C_w$, and the standard object $M_w$ and  costandard objects $A_w$ are the $!$-extension and $*$-extension of the constant perverse sheaf on $C_w$, respectively. 

In favorable situations, e.g. when the strata are all affine spaces $\mathbb{A}^{\ell_w}$, then $\Ext$s between perverse sheaves smooth along the strata coincide with $\Ext$s between their underlying constructible complexes, cf. \cite{tilt}. In particular, for any object $N$ of $\A$, by adjunction $\RHom(N, A_w)$ counts the stalks of $N$ along $C_w$, and $\RHom(M_w, N)$ counts the costalks of $N$ along $C_w$. 

In this setup, the derived BGG resolution can be pictured as follows. The first triangle \[ N_1 \rightarrow \bigoplus_{w \in \W_{\ell_\circ}} M_w \otimes \operatorname{RHom}(M_w, N) \rightarrow N \xrightarrow{+1}\]clears out the stalks of $N$ along top dimensional strata, and the homotopy kernel $N_1$ carries, up to a shift, the same stalks of $N$ along strata of codimension at least one. Similarly, the triangle \[ N_2 \rightarrow \bigoplus_{ w \in \W_{\ell_\circ - 1}} M_w \otimes \operatorname{RHom}(M_w, N_1) \rightarrow N_1 \xrightarrow{+1}\]clears out the stalks of $N_1$ along codimension one strata, and $N_2$ stores the stalks of $N$ along strata of codimension at least two, etc. When one has finally cleared out the stalks of $N$ along the lowest dimensional strata, the homotopy kernel $N_{\ell_\circ + 1}$ has no stalks, i.e. vanishes, and the resolution terminates.\label{int}
\end{re}

\section{Ordinary resolutions}
\label{ordord}
  Let us call a complex of the form introduced in Section \ref{ordres}, i.e. \[0 \rightarrow \bigoplus_{w \in \W_{\ell^\circ}} M_w \otimes V_w \rightarrow \bigoplus_{w \in \W_{\ell^\circ + 1}} M_w \otimes V_w \rightarrow \cdots \rightarrow \bigoplus_{w \in \W_{\ell_\circ}} M_w \otimes V_w \rightarrow 0,\]where $V_w$ are objects of $\operatorname{Vect}^\heartsuit$, and $M_w$ occurs in cohomological degree $\ell_w$, a BGG complex. Unliked the derived resolutions of the preceding section, a BGG complex can be considered as a single object of $\A$. Recall from Section \ref{ordres} that we say an object $N$ of $\A$ admits an ordinary resolution if $N$ is quasi-isomorphic to a shift of a BGG complex. 

 To fix ideas, let us unwind what it means for an object of the abelian category $\A^\heartsuit$ to admit an ordinary resolution.
\begin{pro} Let $N$ be an object of $\A^\heartsuit$, viewed as an object of $\A$ in cohomological degree zero. Then $N$ admits an ordinary resolution if and only if, for some integer $d$,  there exists an exact sequence of the form \begin{equation} \label{naive} 0 \rightarrow \bigoplus_{w \in \W_{\ell^\circ}} M_w \otimes V_w  \rightarrow \bigoplus_{w \in \W_{\ell^\circ + 1}} M_w \otimes V_w \rightarrow \cdots \rightarrow \bigoplus_{w \in \W_d} M_w \otimes V_w \rightarrow  N \rightarrow 0.\end{equation}

\begin{proof} Given a sequence of the form \eqref{naive}, we can think of the last arrow as a quasi-isomorphism from a shifted BGG complex to $N$. On the other hand, suppose there exists a BGG complex $\mathscr{C}$ whose cohomology is $N[-d]$, for some integer $d$. Noting that, for any integer $n$, a map from a sum of standard modules $M_w$ with $\ell_w = n$ to a nonzero sum of standard modules $M_y$ of length $\ell_y = n + 1$ cannot be surjective, it follows that the complex $\mathscr{C}$ has no nonzero terms in degrees greater than $d$. This yields the desired sequence \eqref{naive}.  \end{proof} \end{pro}

Let us now obtain a characterization of objects admitting ordinary resolutions. Up to taking cohomological shifts, it suffices to characterize objects quasi-isomorphic to a BGG complex. 

\begin{theo} Let $N$ be an object of $\A$. Then $N$ is quasi-isomorphic to a BGG complex if and only if for every $w \in \W$ one has the vanishing \begin{equation}\label{van} \Ext^n(N, A_w) = 0, \quad \quad \forall n \neq -\ell_w.\end{equation} \label{ordresobj}
\end{theo}

\begin{proof} If $N$ admits an ordinary resolution, we will deduce the concentration condition \eqref{van} from the vanishing \eqref{rhomm}. Indeed, we may assume $N$ is a BGG complex, and as one can compute derived functors using acyclic resolutions, we need only to understand the cohomology of the complex $\Hom(N, A_w)$. As the latter complex is nonzero only in degree $-\ell_w$, we are done. 

Next suppose that $N$ satisfies the vanishing condition \eqref{van}. We will construct an ordinary resolution from the derived resolution. As we proved in Theorem \ref{derbgg}, the multiplicity of the standard object $M_w$ appearing in the derived resolution is the complex$$(\mathbb{D} \RHom(N, A_w))[\ell_w - \ell_\circ],$$which by our assumption on $N$ is concentrated in cohomological degree $\ell_\circ$. We are then done by the following general observation, which we record as a separate lemma.   \end{proof}

\begin{lem} Let $\mathscr{A}^\heartsuit$ be an abelian category, and $\mathscr{A}$ its bounded derived category. Suppose we are given a sequence of distinguished triangles \begin{align*} A_1 \rightarrow &B_0 \rightarrow A \xrightarrow{+1}, \\ A_2 \rightarrow &B_1 \rightarrow A_1 \xrightarrow{+1}, \\ &\hspace{1mm}\vdots \\ A_{n} \rightarrow &B_{n-1} \rightarrow A_{n-1} \xrightarrow{+1}, \\ 0 \rightarrow &B_n \rightarrow A_n \xrightarrow{+1} . \end{align*}If the $B_i$ are all objects of $\mathscr{A}$ in cohomological degree zero, $0 \leqslant i \leqslant n$, then $A$ is quasi-isomorphic to a complex of the form $$\cdots \rightarrow B_2 \rightarrow B_1 \rightarrow B_0,$$where $B_i$ is in cohomological degree $-i$ and the morphisms are the compositions $B_i \rightarrow A_i \rightarrow B_{i-1}$. 
\label{hom}
\end{lem}

\begin{proof} We induct on $n$. Consider the triangle $$A_1 \rightarrow B_0 \rightarrow A_0 \xrightarrow{+1}.$$By inductive hypothesis, we may identify  $A_1$ with a complex of the form $$\cdots \rightarrow B_3 \rightarrow B_2 \rightarrow B_1.$$
Using the usual $t$-structure on $\mathscr{A}$, we have $\Hom(A_1, B_0) \simeq \Hom(\tau^{\geqslant 0} A_1, B_0).$ But the latter is the cokernel $C$ of $B_2 \rightarrow B_1$, concentrated in cohomological degree zero. As $\Hom(C, B_0)$ coincides for $\mathscr{A}^\heartsuit$ and $\mathscr{A}$, we may write $A_1 \rightarrow B_0$ as an honest map of complexes $$\xymatrix{ \cdots \ar[r] &  B_3 \ar[r] & B_2 \ar[r] & B_1 \ar[d] \\ && & B_0.}$$Computing the cone using this model for the morphism finishes the induction. More carefully, this introduces a sign on all the differentials $B_{i+1} \rightarrow B_{i}$ for $i \geqslant 1$, but the resulting complex is isomorphic to the one without any signs. \end{proof}

Corollaries \ref{cor1}, \ref{cor2}, which we repeat here for the reader's convenience, follow immediately from Theorem \ref{ordresobj}. 

\begin{cor} If an object $N$ of $\A$ admits an ordinary BGG resolution, so does any summand of $N$.  
\end{cor}

\begin{cor} Consider a distinguished triangle in $\A$ $$N' \rightarrow N \rightarrow N'' \xrightarrow{+1}.$$If $N'$ and $N''$ are quasi-isomorphic to BGG complexes, then so is $N$. 
\end{cor}

\section{The constructible {$t$}-structure}

Let $\A$ be as before. Our goal in this section is to set up an exotic $t$-structure on $\A$, whose heart will be objects admitting an ordinary resolution. We should say at the outset that this was previously done by Cline--Parshall--Scott \cite{scott2}. However, as we will desire descriptions of the heart and coconnective objects which we were unable to locate in {\em loc. cit.}, we provide a detailed construction.

Recall that a $t$-structure may be reconstructed from its category of coconnective objects. Accordingly, let us define $\A^{\geqslant 0}$ to be the full subcategory of $\A$ with objects $N$ satisfying, for each $w \in \W$, the vanishing condition\begin{equation} \label{geq}\Ext^n(N, A_w) = 0, \quad \quad \forall n > - \ell_w.\end{equation}We will refer to such an object as coconnective.
\begin{re}In the motivating case of perverse sheaves smooth with respect to a stratification by affine spaces $\mathbb{A}^{\ell_w}, w \in \W$, the costandard object $A_w$ is explicitly given by the $*$-extension of $\underline{k}[\ell_w]$, where $\underline{k}$ denotes the constant sheaf. It follows that $\A^{\geqslant 0}$ identifies with objects of the constructible derived category whose stalks all lie in nonnegative cohomological degrees, i.e. the nonnegative part of the usual $t$-structure.   \end{re}

Let us collect some first observations concerning this subcategory. 
\begin{lem}The following properties of $\A^{\geqslant 0}$ hold.  \label{duh}

\begin{enumerate}
    \item If $N$ is coconnective, then so is $N[-1]$.
    \item Given a distinguished triangle $N' \rightarrow N \rightarrow N'' \xrightarrow{+1},$ if $N'$ and $N''$ are coconnective, then so is $N$. 
    \item For any $w \in \W$, the object $M_w[-\ell_w]$ is coconnective. 
\end{enumerate}
\end{lem}
\begin{proof} Items (1) and (2) are immediate, and (3) follows from Equation \eqref{rhomm}. \end{proof}

As before, let $Z$ denote a closed subset of $\W$ with open complement $U$, with inclusions $i$ and $j$, respectively. We define $\A_Z^{\geqslant 0}$ to be the full subcategory of $\A_Z$ consisting of objects satisfying  \eqref{geq} for $w \in Z$, and we similarly define $\A_U^{\geqslant 0}$. With this, we have the following:

\begin{lem} The following properties of coconnective objects hold.  \label{duhhh}

\begin{enumerate}
    \item An object $N$ of $\A$ is coconnective if and only if $j^*N$ and $i^*N$ are. 
    \item An object $N_U$ of $\A_U$ is coconnective if and only if $j_! N_U$ is. 
    \item An object $N_Z$ of $\A_Z$ is coconnective if and only if $i_* N_Z$ is. 
\end{enumerate}
\end{lem} 
\noindent The proof of Lemma \ref{duhhh} is straightforward. Less immediately, we also have

\begin{pro}\label{indst} Suppose $U$ consists of a single element $u$ of maximal length. Then an object $N$ of $\A$ is coconnective if and only if $j^* N$ and $i^! N$ are. 
\begin{proof}To show the `only if' implication, we already saw that $j^* N$ is coconnective. To see that $i^! N$ is, consider the distinguished triangle \[i^! j_! j^* N \rightarrow i^! N \rightarrow  i^! i_*i^* N \xrightarrow{+1}.\]By Lemma \ref{duh}, it suffices to show the coconnectivity of the outer two objects. As $i^! i_* i^* N \simeq i^* N$, its coconnectivity is clear. Note that $j_! j^* N$ is of the form $M_u \otimes V_u$, for a vector space $V_u \in \operatorname{Vect}$ concentrated in degrees at least $-\ell_w$. Accordingly, to show the coconnectivity of $i^! j^! j^* N$ it suffices to show the coconnectivity of $i^! M_u [-\ell_u]$, or equivalently that of $i_* i^! M_u [-\ell_u]$ by Lemma \ref{duhhh}. To see this, applying the triangle $$i_* i^! \rightarrow \operatorname{id} \rightarrow j_* j^* \xrightarrow{+1}$$ to $M_u[-\ell_u]$, we obtain a triangle of the form\begin{equation}\label{wts}i_* i^! M_u [-\ell_u] \rightarrow M_u[-\ell_u] \rightarrow A_u[-\ell_u] \xrightarrow{+1}.\end{equation}We claim that $A_u[-\ell_u]$ is coconnective. By definition, we must show for any $w \in \W$ the vanishing \[\Ext^n( A_u, A_w) = 0, \quad \quad \forall n > \ell_u - \ell_w.\]However, it is standard that $A_w$ has injective dimension at most $\ell_u - \ell_w$, as follows from the injectivity of $A_u$ and a descending induction using the dual of Definition \ref{dhw}(3). The coconnectivity of $i_* i^! M_u[-\ell_u]$ now follows by applying $\operatorname{RHom}(-, A_w)$ to the triangle \eqref{wts}, and using the coconnectivity of $A_u$ and Equation \eqref{rhomm}. 

Conversely, suppose that $j^* N$ and $i^! N$ are coconnective, and consider the triangle $$i_* i^! N \rightarrow N \rightarrow j_* j^* N \xrightarrow{+1}.$$By Lemma \ref{duh}, it suffices to show the coconnectivity of the outer two objects. The assertion for $i_* i^! N$ is clear. For the other, by assumption $j_* j^* N$ is of the form $A_u \otimes V$ for a vector space $V \in \operatorname{Vect}$ concentrated in degrees at least $-\ell_u$. As we showed the coconnectivity of $A_u[-\ell_u]$ above, we are done. \end{proof}\end{pro}

Recall that $\A^{\geqslant 0}$ defines a $t$-structure if (1) the inclusion $\A^{\geqslant 0} \rightarrow \A$ admits a left adjoint $\tau^{\geqslant 0}$. Moreover, (2) if we write $\A^{< 0}$ for the full subcategory of $\A$ consisting of objects $M$ satisfying $\operatorname{Hom}(M, N) \simeq 0$ for every coconnective object $N$, then the inclusion $\A^{< 0}$ admits a right adjoint $\tau^{ < 0}$, and (3) the tautological counit and unit maps form a distinguished triangle \[\tau^{ < 0} \rightarrow \operatorname{id} \rightarrow \tau^{\geqslant 0} \xrightarrow{+1}.\]We may now show that:

\begin{cor} \label{tstruc}The coconnective objects $\A^{\geqslant 0}$ define a $t$-structure on $\A$. \end{cor}

\begin{proof} We proceed by induction on the size of $\W$. When $\W$ is a singleton $u$, under the canonical identification $\A \simeq \operatorname{Vect}$ exchanging $L_u$ and $k$, we are simply obtaining a shift of the standard $t$-structure on $\operatorname{Vect}$. For the inductive step, decompose $\W$ into $Z$ and $u$ as in Proposition \ref{indst}. By our inductive hypothesis, $\A_Z^{\leqslant 0}$ and $\A_U^{\leqslant 0}$ define $t$-structures on $\A_Z$ and $\A_U$. By Proposition \ref{indst}, $\A^{\geqslant 0}$ coincides with the coconnective objects of the gluing of the $t$-structures on $\A_Z$ and $\A_U$, cf. Theorem 1.4 of \cite{bbd}, as desired. \end{proof}

In the remainder of this section, we explicitly identify $\A^{< 0}$ and the heart \[\A^\heartsuit_{con} := \A^{ \leqslant 0} \cap \A^{\geqslant 0}.\]Let us call an object $M$ of $\A^{< 0}$ connective. 

\begin{pro} An object $M$ of $\A$ is connective if and only if for every $w \in \W$, one has the vanishing \begin{equation} \label{conn}\Ext^n( M, A_w) = 0, \quad \quad \forall n \leqslant -\ell_w.\end{equation}

\end{pro}

\begin{proof}Let us show every connective object $M$ satisfies Equation \eqref{conn}. Indeed, suppose that \eqref{conn} failed to hold for some fixed $y \in \W$. Consider the closed set $i: Z \rightarrow \W$, where $Z = \{ w \in W: w \leqslant y \}$. By assumption, the restriction of $i^* M$ to the open stratum $y$ of $Z$ is not connective, i.e. we have a nontrivial map $$i^* M \rightarrow A_y[-\ell_y - d],$$for some nonnegative integer $d$. It follows that the composite $$M \rightarrow i_* i^* M \rightarrow A_y [-\ell_y - d]$$is nonzero, and the last appearing object is coconnective, as follows from the proofs of Proposition \ref{indst} and Corollary \ref{tstruc}, which contradicts the connectivity of $M$. 

For the converse, suppose $M$ satisfies \eqref{conn}, $N$ is coconnective, and we are given a map $f: M \rightarrow N$. Fixing a maximal element $j: u \rightarrow \W$, by our assumptions the map $j^*f: j^*M \rightarrow j^* N$ is trivial, as there are no cohomological degrees in which the objects of $\operatorname{Vect}$ corresponding to $j^* M$ and $j^* N$ are both nontrivial. By applying $\operatorname{RHom}(M, -)$ to the distinguished triangle \[ i_! i^* N \rightarrow N \rightarrow j_* j^* N \xrightarrow{+1}\]and using the just noted vanishing of $\operatorname{Hom}( M, j_* j^* N)$ and $\operatorname{Hom}(M, j_* j^* N [-1])$, it follows that $f$ factors (uniquely) as $M \rightarrow i_* i^! N \rightarrow N.$ We may further factor this as $$M \rightarrow i_* i^* M \rightarrow i_* i^! N \rightarrow N.$$From the definition, it is straightforward that $i^* M$ is connective, and we saw the coconnectivity of $i^!N$ in Proposition \ref{indst}, hence we are done by induction on the size of $\W$. \end{proof}

As a corollary, we obtain an very explicit description of the heart. 

\begin{cor} \label{esssurj}An object $M$ of $\A$ belongs to $\A^{\heartsuit}_{con}$ if and only if for every $w \in \W$ one has the vanishing $$\Ext^n(M, A_w) = 0, \quad \quad \forall n \neq -\ell_w.$$\end{cor}

In particular, an object lies in the heart if and only if it admits an ordinary resolution. We will upgrade this statement to an equivalence of categories presently. To do so, consider $\operatorname{Ch}(\A^\heartsuit)$, the abelian category of chain complexes with entries in $\A^\heartsuit$, and set $\operatorname{BGG}$ to be the full subcategory of $\operatorname{Ch}(\A^\heartsuit)$ with objects the BGG complexes, cf. Section \ref{ordord}. Note that $\operatorname{BGG}$ is an abelian subcategory of $\operatorname{Ch}(\A^\heartsuit)$, with kernels and cokernels of morphisms agreeing with those computed in the ambient category $\operatorname{Ch}(\A^\heartsuit)$, i.e. computed in each cohomological degree.

\begin{theo}The tautological composition $\operatorname{BGG} \rightarrow \operatorname{Ch}(\A^\heartsuit) \rightarrow \A$ induces an equivalence \[ \operatorname{BGG} \simeq \A^\heartsuit_{con}.\]\end{theo}

\begin{proof}Corollary \ref{esssurj} yields that the map $\operatorname{BGG} \rightarrow \A$ factors through $\A^\heartsuit_{con}$, and is moreover essentially surjective. To see fully faithfulness, let $M$ and $M'$ be two BGG complexes. We must show that any morphism $f: M \rightarrow M'$ in $\A$ corresponds to a unique morphism in $\operatorname{Ch}(\A^\heartsuit)$. 

Write $V_w$, for $w \in \W$, for the multiplicity vector space of $M_w$ appearing in $M$, and similarly define $V'_w$. Notice that a map of chain complexes $M \rightarrow M'$ is determined by the associated morphisms $V_w \rightarrow V'_w,$ for $w \in \W$, induced by $\Hom(M_w \otimes V_w, M_w \otimes V'_w) \simeq \Hom(V_w, V'_w)$. Thus the uniqueness of such a morphism $M \rightarrow M'$ is forced by applying $\RHom(-, A_w)$ to $f$, and using the identification of $V_w, V'_w$ with $\Ext^n(M, A_w)^\vee, \Ext^n(M', A_w)^\vee$, for appropriate $n \in \Z$. 

That such a morphism exists follows from the functoriality of derived resolutions proved in Theorem  \ref{derbgg} and an iterated coning argument similar to that of Lemma \ref{hom}. \end{proof}

Let us conclude this section with a corollary concerning the spectral sequence \eqref{sqsq}. Informally, it says the spectral sequence goes from the constructible cohomology on its first page to the perverse cohomology on its last page.

\begin{cor} For any $N$ in $\A$, the rows of \eqref{sqsq} are the constructible cohomology sheaves of $N$. I.e., the complex $E_1^{*, q}$ is an ordinary resolution of $H_{con}^{\ell_\circ + q}(N)$, where for an integer $i$ we denote by $H_{con}^i$ the constructible cohomology functor \[ [-i]\tau^{\leqslant i} \tau^{\geqslant i}: \A \rightarrow \A^\heartsuit_{con}.\]\end{cor}

\section{Highest weight modules admitting ordinary resolutions}

Let $X$ be a separated complex variety with a stratification by locally closed subvarieties $C_w,$ indexed by $w \in \W$. Write $X_w$ for the closure of $C_w$ and $d_w$ for its dimension. For a field $k$, write $\bk$ for the corresponding constant sheaf on $X$. Let $D^b(X)$ denote the bounded derived category of sheaves of $\bk$-modules on $X$ in the analytic topology. Let $D^b_\W(X)$ denote the full subcategory of $D^b(X)$ with $\W$-constructible cohomology. I.e.,  an object of $D^b(X)$  lies in $D^b_\W(X)$ if and only if its cohomology sheaves are, after $*$-restriction to any stratum  $C_w$, local systems of $\underline{k}$-modules of finite rank. Finally, set $\A^\heartsuit$ to be the full subcategory of $D^b(X)$ of objects which are perverse and lie in $D^b_\W(X)$. 

Let us now impose some constraints on the stratification which are satisfied in common geometric representation theoretic situations. Namely, let us assume that each stratum $C_w$ is contractible in the analytic topology, and the inclusions $\iota^w: C_w \rightarrow X,$ for $w \in\W,$ are affine morphisms. Then the simple objects of $\A$ are the intersection complexes $$L_w := \iota^w_{!*} \bk[d_w], \quad \quad \text{ for }w \in \W.$$Using the partial order on $\W$ induced by the closure relation on strata, $\A$ is a highest weight category, with standard objects $M_w := \iota^w_! \bk[d_w]$ and costandard objects $ A_w := \iota^w_* \bk[d_w]$. 
Moreover, the inclusion $\A^\heartsuit \rightarrow D^b_\W(X)$ can be prolonged to an equivalence $ \A \simeq D^b_\W(X)$, cf.  
\cite{tilt}.  We will define the compatible length function by $\ell_w = d_w,$ for $w \in \W$. 

The question we study in this section is: when does a highest weight sheaf admit an ordinary resolution? Here by highest weight sheaf we mean a perverse sheaf $M$ in $\A$ admitting a surjection $M_w \twoheadrightarrow M$, for some necessarily unique $w \in \W$. We first determine what shifted sheaf $M$ could possibly be. 

\begin{theo} Fix $w \in \W$, and let $M_w \twoheadrightarrow M$ be a highest weight sheaf. Suppose $(\star)$ for each $y \leqslant w$, where $d_w - d_y \leqslant 1$, that $j_*\bk_{C_y} \simeq \bk_{X_y}$, i.e. the links to strata within $X_w$ and its divisors $X_y$ are connected.  
Then the lowest cohomology sheaf of $M$ is the $!$-extension of the constant sheaf on $X_w$ minus certain boundary divisors $X_{z},$ i.e. $z \leqslant w$, and $d_w - d_z = 1$. 
\label{bs}

\end{theo}

\begin{proof} From the surjection $M_w \twoheadrightarrow M$, we obtain that $\cosoc M \simeq \cosoc M_w = L_w$. Consider the corresponding short exact sequence $$0 \rightarrow K \rightarrow M \rightarrow L_w \rightarrow 0.$$As $[K:L_w] = 0$, it follows that $K$ is supported in codimension at least one. Taking constructible cohomology, we obtain by $(\star)$ an exact sequence \begin{equation} \label{uno}0 \rightarrow \mathscr{H}^{-d_w} M \rightarrow \bk_{X_w} \rightarrow \mathscr{H}^{-d_w + 1} K.\end{equation}We next look at $K$, a perverse sheaf supported on $X_w \setminus C_w =: \p X_w$. Let us decompose  $\p X_w = U \sqcup Z$, where $j: U \rightarrow \p X_w$ is the disjoint union of the open divisor strata $C_y,$ where $y \leqslant w$, and $d_w - d_y = 1$. Let us look at the corresponding standard distinguished triangle \begin{equation} \label{dos} i_! i^! K \rightarrow K \rightarrow j_* j^* K \xrightarrow{+1}.\end{equation} The closed set term $i_! i^! K$ is supported in codimension two and lies in nonnegative perverse degrees. The open restriction $j^* K$ must be, by perversity and constructibility, a sum $$j^* K \simeq \bigoplus_{\substack{y \leqslant w: \\ d_w - d_y = 1}} \bk_{C_y} \otimes V_y[d_y],$$for some multiplicity spaces $V_y$. By $(\star)$, it then follows that $$\mathscr{H}^{-d_w + 1} j_* j^* K \simeq \bigoplus_{\substack{y \leqslant w: \\ d_w - d_y = 1}}  \bk_{X_y} \otimes V_y.$$Combining these analyses, the bottom of the long exact sequence in cohomology sheaves for \eqref{dos} reads as \begin{equation} \label{tres} 0 \rightarrow \mathscr{H}^{-d_w + 1} K \rightarrow \bigoplus_{\substack{y \leqslant w: \\ d_w - d_y = 1}} \bk_{X_y} \otimes V_y.\end{equation}Combining \eqref{uno} and \eqref{tres} gives an exact sequence \begin{equation} 0 \rightarrow \mathscr{H}^{-d_w} M \rightarrow \bk_{X_w} \rightarrow \bigoplus_{\substack{y \leqslant w: \\ d_w - d_y =1}} \bk_{X_y} \otimes V_y,\end{equation}from which the desired claim follows. \end{proof}

The analysis so far uses only the zeroth step of the resolution $M_w \twoheadrightarrow M$. We next study the possible first steps of the resolution, i.e. we need to understand $\Hom(M_y, M_w)$, for $d_w - d_y = 1$. 

\begin{lem} Suppose that $y \leqslant w, d_w - d_y = 1$. If $j_* \bk_{C_w} \simeq \bk_{X_w}$, then $\Hom(M_y, M_w)$ is one dimensional.  \label{homs}
\end{lem}

\begin{proof} Write $\mathbb{D}$ for the Verdier duality, which preserves $\A$. Since  $C_w,C_y$ are smooth, we have $$\RHom(M_y, M_w) \simeq \RHom(\mathbb{D} M_w, \mathbb{D}M_y) \simeq \RHom(A_w, A_y).$$By adjunction, $\Hom(A_w, A_y) \simeq \Hom( H^1(L, k), k)$,  where $L$ is the link of $C_y$ in $X_w$. This is a compact one-dimensional manifold, i.e. a disjoint union of circles. As $j_* \bk_{C_w} \simeq \bk_{X_w}$, it is precisely one circle, as desired.  \end{proof}

The preceding lemma affords us a parametrization of possible highest weight objects $M_w \twoheadrightarrow M$ admitting ordinary resolutions. Namely, write $D_w = \{ y \leqslant w: d_w - d_y = 1\}$. For a subset $I \subset D_w$, consider the object $M_I$ obtained via the exact sequence $$ \bigoplus_{ y \in I} M_y \otimes \Hom(M_y, M_w) \rightarrow M_w \rightarrow M_I \rightarrow 0. $$It is clear any $M_w \twoheadrightarrow M$ admitting an ordinary resolution must be isomorphic to $M_I$ for some $I$. Indeed, to compute the zeroth term of our putative BGG complex, we use the injection $\Hom(M, A_x) \rightarrow \Hom(M_w, A_x),$ for $x \in \W$, to conlude the first term must be $M_w$ with multiplicity one. $M$ would then be determined as the cokernel of the next step in the resolution, i.e. must be of the form $M_I$ for some $I$. That these are all distinct is a consequence of the following lemma, which describes the $M_I$ modulo behavior in codimension two. 

\begin{lem} For any $I \subset D_w$, consider the open subvariety of $X_w$ given by$$U_I := C_w \sqcup \bigsqcup_{y \in I} C_y.$$If $j_* \bk_{C_w} \simeq \bk_{X_w}$, then the restriction of $M_I$ to $U_{D_w}$ is the $!$-extension of the constant sheaf on $U_{D_w \setminus I}$. \label{poles}
\end{lem}

\begin{proof} First note that Lemma \ref{homs} implies that $[M_w: L_y] = 1$, for all $y \in D_w$. Indeed, consider the short exact sequence $$0 \rightarrow K \rightarrow M_w \rightarrow L_w \rightarrow 0.$$We have $\Hom(M_y, M_w) \simeq \Hom(M_y, K)$. As $K$ is supported in codimension one, where $y$ is maximal, $\dim \Hom(M_y, K) = [K: L_y]$. By considering Jordan-H\"older content, it follows that the restriction of $M_I$ to each $U_y, y \in D_y$, is isomorphic to the restriction of either $L_w$ or $M_w$. By the assumption on $j_* \bk_{C_w}$, the restriction of $L_w$ to $U_{D_y}$ is $\bk[d_w]$, and the claim follows easily. \end{proof}

\begin{re} Informally, the perverse sheaf $M_w$ has poles along the boundary divisors of $X_w$. For $M_I$, the addition of each $y \in I$ removes the pole along the divisor $X_y$. \end{re}

Combining the previous results in this section, we obtain:

\begin{theo} Fix $w \in \W$, and suppose the connectivity condition $(\star)$ of Theorem \ref{bs} holds. For $I \subset D_w$, consider the open embedding $j: X_w \setminus \cup_{y \notin I} X_y \rightarrow X_w$. The following are equivalent:

\begin{enumerate}
\item $M_I$ admits an ordinary resolution.

\item The sheaf $j_! \bk[d_w]$ is perverse.

\item There is an isomorphism of constructible complexes $M_I \simeq j_! \bk[d_w]$. 

\end{enumerate}
\label{topc}

\end{theo}

\begin{proof}  By Theorem \ref{ordresobj}, we know that if $M_I$ admits an ordinary resolution then it is a shifted sheaf. In particular it would coincide with its cohomology sheaf $\mathscr{H}^{-d_w} M_I$. By Theorem \ref{bs} this sheaf is the $!$-extension of the constant sheaf from the complement of certain divisors $X_y, y \in D_w$. To determine which divisors we can ignore everything in codimension two. Lemma \ref{poles} then implies $\mathscr{H}^{-d_w} M_I \simeq j_! \bk[d_w]$. We have shown the equivalence of (1) and (3) and the implication (1) implies (2). 

To see that (2) implies (1), note $j_! \bk[d_w]$ admits an ordinary resolution by Theorem \ref{ordresobj}. If $j_! \bk[d_w]$ is moreover perverse, the resolution takes the form in Equation \eqref{naive}. We may always read the terms of the resolution from the stalks of $j_! \bk[d_w]$. Looking at the zeroth and first terms shows $j_! \bk[d_w]$ must be $M_I$. \end{proof}

In the following corollaries, it is understood that condition $(\star)$ holds. 

\begin{cor} For any $w \in \W$, the intersection complex $L_w$ on $X_w$ admits an ordinary resolution if and only if $X_w$ is $k$-smooth. \label{c1}
\end{cor}

\begin{proof} If $X_w$ is $k$-smooth, $L_w \simeq \bk_{X_w}[d_w]$, and hence admits an ordinary resolution by Theorem \ref{ordresobj}. If $L_w$ admits an ordinary resolution, by Theorem \ref{topc} we have $L_w \simeq M_{D_w} \simeq \bk_{X_w}[d_w]$. \end{proof}

\begin{cor} Fix $w \in \W$, and $I \subset I' \subset D_w$. If $M_{I'}$ admits a BGG resolution, so does $M_I$.  \label{c2}
\end{cor}

\begin{proof} Briefly, this follows from the Mayer-Vietoris sequence. In more detail, for any $I'' \subset D_w$, let us name the open subvariety appearing in Theorem \ref{topc} as $$\mathscr{U}_{I''} := X_w \setminus \cup_{y \notin I''} X_y.$$ The Mayer-Vietoris exact sequence of sheaves is in particular a distinguished triangle. Shifting by $d_w$, we obtain \begin{equation}\label{mv} j_! \bk_{\mathscr{U}_\emptyset}[d_w] \rightarrow j_! \bk_{\mathscr{U}_I}[d_w] \oplus j_! \bk_{\mathscr{U}_{I' \setminus I}}[d_w] \rightarrow j_! \bk_{\mathscr{U}_{I'}}[d_w] \xrightarrow{+1}. \end{equation} Note the leftmost term of \eqref{mv} is $M_w$. Considering the long exact sequence on perverse cohomology corresponding to \eqref{mv} and applying Theorem \ref{topc} yields the claim. \end{proof}

\section{A worked example} \label{wex}Let $\g$ be a complex semisimple Lie algebra, and write $\OO_0$ for a regular block of its Category $\OO$. In this subsection we will obtain a classification of the highest weight modules of $\OO_0$ admitting ordinary resolutions for the first nontrivial choice of $\g$. 

Let us explain what we mean by nontrivial. Write $X$ for the flag variety associated to $\g$, and $W$ for the associated Weyl group. Write $X_w$, for $w \in W$, for the Schubert subvarieties of $X$. By Corollaries \ref{c1}, \ref{c2}, if a Schubert variety $X_w$ is rationally smooth, then for every $I \subset D_w$ the module $M_I$ admits an ordinary resolution. This observation accounts for all cases when $\g$ is of rank at most two. Alternatively, in rank at most two the $M_I$ are all parabolic Verma modules. 

We will therefore consider the case of type $A_3$. Write $s,t,u$ for the simple reflections generating $W$, where $s$ and $u$ commute. In this case there are two Schubert varieties which are not rationally smooth. These are of dimension four and five, and the corresponding elements of $W$ may be characterized as the unique elements of their length fixed by the involution switching $s$ and $u$. Explicitly, they may be written as the reduced expressions $tsut$ and $stuts$. 

Let us first consider the case of $X_{tsut}$. Consider the spectral sequence \eqref{sqsq} converging to $L_{tsut}$. For transparency of indexing, we will run the spectral sequence in the closed subcategory $\A_{\leqslant tsut}$. In this case, on the $E_1$ page column $-p$ stores the stalks on strata of codimension $p$ in $X_{tsut}$, and row $q$ stores the cohomology sheaf $\mathscr{H}^{-4 + q}$. As the only nontrivial Kazhdan-Lusztig polynomials $P_{y, tsut}$ are $P_{e, tsut} = P_{t, tsut} = 1 + q$, the $E_1$ page looks as follows: 

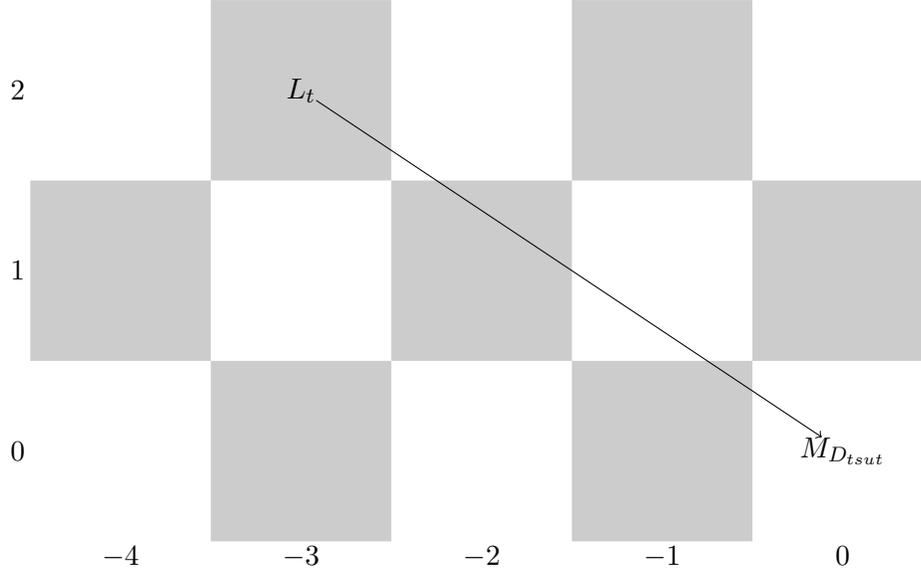
\begin{sseq}[grid=chess,labelstep=1,
entrysize=2.4cm]{-4...0}{0...2}
\ssmoveto{-4}{0}
\ssdrop{M_e}  
\ssmove 1 0 
\ssdrop{\bigoplus_{\substack{y \leqslant w: \\ \ell_w = 1}} M_w} \ssstroke[arrowto]
\ssmove 1 0 
\ssdrop{\bigoplus_{\substack{y \leqslant w: \\ \ell_w = 2}} M_w} \ssstroke[arrowto]
\ssmove 1 0 
\ssdrop{\bigoplus_{\substack{y \leqslant w: \\ \ell_w = 3}} M_w} \ssstroke[arrowto]
\ssmove 1 0 
\ssdrop{M_{tsut}} \ssstroke[arrowto]
\ssmoveto{-4}{2} 
\ssdrop{M_e}
\ssmove 1 0
\ssdrop{M_t} \ssstroke[arrowto]
\end{sseq}

As $\mathscr{H}^{-2} L_{tsut} \simeq \C_{X_t}$, we recognize the complex $M_e \rightarrow M_t$ as the ordinary resolution of $L_t$. It follows the only remaining nontrivial differential in the spectral sequence occurs on $E_3$:

\begin{sseq}[grid=chess,labelstep=1,
entrysize=2.4cm]{-4...0}{0...2}
\ssmoveto{-3}{2}
\ssdrop{L_t}
\ssmoveto{0}{0}
\ssdrop{M_{D_{tsut}}} \ssstroke[arrowto]
\end{sseq}

 We deduce that the bottom row of the $E_1$ page was an ordinary resolution of $M_{D_{tsut}}$. By Corollary \ref{c2}, we conclude every $M_I$ of highest weight $tsut$ admits an ordinary resolution.
 
Further, we may deduce that the kernel of $M_{tsut} \rightarrow L_{tsut}$ is not generated by highest weight vectors. Namely, from the classification of homomorphisms between Verma modules, any nontrivial homomorphism from a Verma module to $M_{tsut}$ would factor through an inclusion $M_y \rightarrow M_{tsut}$, for $y$ a divisor of $tsut$. This affords an example of a simple module which cannot be quasi-isomorphic to a complex of Verma modules, as we spell out in the following:
 
 \begin{lem} Suppose $\A^\heartsuit$ is a highest weight category, and $L_w$ a simple object such that the kernel of $M_w \rightarrow L_w$ is not generated by highest weight vectors, i.e. not a quotient of a sum of standard objects. Then $L_w$ is not quasi-isomorphic to a complex whose terms are sums of Verma modules. \end{lem}

\begin{proof} Suppose $\mathscr{C}$ was a complex whose terms were sums of Verma modules, and whose cohomology consisted of $L_w$ in degree zero. Let us say that an element $y \in \W$ occurs in $\mathscr{C}$ if $M_y$ is a summand in a term of $\mathscr{C}$. Pick a $y$ maximal in $\W$ such that $y$ occurs in $\mathscr{C}$, and note that the factors of $M_y$ in each degree form a subcomplex $\mathscr{C}^y$ of $\mathscr{C}$. If $y$ is not less than or equal to $w$, then it follows that $\mathscr{C}^y$ is acyclic, and we may replace $\mathscr{C}$ by $\mathscr{C}/\mathscr{C}^y$. Iterating this, we may assume that $w$ is the unique maximum of the subset of $\W$ occuring in $\mathscr{C}$. By the same argument, we may further assume that $M_w$ appears in exactly one degree, which is necessarily zero. It follows that $L_w$ is isomorphic to the quotient of $M_w$ by $S$, its intersection with the image of $\mathscr{C}^{-1}$. Let us write $\mathscr{C}^0$ as $M_w \oplus M$, where $M$ is a sum of Verma modules $M_y,$ for $y < w$. As $S$ is a submodule of the image of the composition $$\mathscr{C}^{-1} \rightarrow M_w \oplus M \rightarrow  M_w,$$which again must be a proper submodule $S'$ of $M_w$, it follows that $S = S'$. In particular, $S$ is generated by highest weight vectors, which is a contradiction. \end{proof}

The analysis of $X_{stuts}$ is similar. Here one uses that $\mathscr{H}^{-5} L_{stuts} \simeq \C_{X_{stuts}}, \mathscr{H}^{-3} L_{stuts} \simeq \C_{X_{su}}$. As the latter sheaf is, up to a shift, $L_{su}$, the same reasoning about the $E_1$ page and subsequent differentials applies. In particular, every $M_I$ of highest weight $stuts$ admits an ordinary resolution. 

This completes the classification. Adding things up, we have shown there are up to isomorphism 155 highest weight modules in type $A_3$ admitting ordinary resolutions. Previous examples of such modules were to the author's knowledge limited to rationally smooth simple modules and parabolic Verma modules. The latter correspond to considering divisors lying in the left descent set of $w$ rather than all divisors. In particular, in total these account for only 75 + 22 - 8 = 89 of the above modules. 

 \bibliographystyle{plain}    
\bibliography{output}

\end{document}